\documentclass[11pt]{amsart}
\usepackage{amscd, amsmath, amsthm, amssymb}
\usepackage{color}
\usepackage{graphicx}
\usepackage{mathrsfs}
\usepackage{verbatim}
\usepackage[all]{xy}
\usepackage{hyperref}

\newtheorem{theorem}{Theorem}[section]
\newtheorem{lemma}[theorem]{Lemma}

\newtheorem{conjecture}[theorem]{Conjecture}

\newtheorem*{conj}{Conjecture}  

\theoremstyle{definition}
\newtheorem{definition}[theorem]{Definition}
\newtheorem{example}[theorem]{Example}

\theoremstyle{remark}
\newtheorem{remark}[theorem]{Remark}

\numberwithin{equation}{section}

\newcommand\bba{\mathbb{A}} 
\newcommand\B{\mathbf{B}}

\newcommand\OO{{\mathcal{O}}} \newcommand\PP{{\mathbb{P}}} 
\newcommand\QQ{{\mathbb{Q}}} \newcommand\RR{{\mathbb{R}}} 
\newcommand\ZZ{{\mathbb{Z}}}

\newcommand\bo{\overline{\mathcal{O}}} 
\newcommand\qb{{\overline{\mathbb{Q}}}}
\newcommand\Supp{{\rm Supp}\,}

\newcommand\ts{\widetilde{S}}

\begin{document}

\title[Periodic points and arithmetic degrees]{Periodic points and arithmetic degrees of certain rational self-maps}

\author[Long\ \ Wang]{Long\ \ Wang}
\address{Graduate School of Mathematical Sciences, The University of Tokyo, 3-8-1 Komaba, Meguro-Ku, Tokyo 153-8914, Japan}
\email{wangl11@ms.u-tokyo.ac.jp}

\begin{abstract} Consider a cohomologically hyperbolic birational self-map defined over the algebraic numbers, for example, a birational self-map in dimension two with the first dynamical degree greater than one, or in dimension three with the first and the second dynamical degrees distinct. We give a boundedness result about heights of its periodic points. This is motivated by a conjecture of Silverman for polynomial automorphisms of affine spaces. We also study the Kawaguchi--Silverman conjecture concerning dynamical and arithmetic degrees for certain rational self-maps in dimension two. In particular, we reduce the problem to the dynamical Mordell--Lang conjecture and verify the Kawaguchi--Silverman conjecture for some new cases. As a byproduct of the argument, we show the existence of Zariski dense orbits in these cases. 
\end{abstract}



\maketitle

\thispagestyle{empty}

\section{Introduction}\label{intr}

Throughout this paper, we work over $\qb$. Let $f: \PP^N \to \PP^N$ be a morphism of degree at least two. Then Northcott showed that the set of (pre-)periodic points for $f$ is a set of bounded height (\cite[Lemma 4]{Nor50}, see also \cite{Le72} and \cite[Theorem 1.3]{Si94}). 
Here $P \in \PP^N(\qb)$ is called $f$-\textit{periodic} (resp. $f$-\textit{preperiodic}), if $f^{n}(P) = P$ for some $n \in \ZZ_{>0}$ (resp. $f^{n}(P) = f^{m}(P)$ for some $n > m$).

More generally, one may ask if there are analogous finiteness results for arbitrary dominant rational self-maps $\PP^N \dashrightarrow \PP^N$. The essential difficulty is that the theory of height functions does not behave functorially with respect to rational maps. In \cite{Si94}, Silverman studied the H\'{e}non map
\[ \bba^2 \to \bba^2, (x, y) \mapsto (y, y^2 + ax + b)  
\]
and proposed the following conjecture as a first step towards a general finiteness conjecture for rational maps.

\begin{conj}[{\cite[Conjecture in Page 239]{Si94}, \cite[Conjecture 7.21]{Si07}}] Let $f^{\circ}: \bba^N \to \bba^N$ be a polynomial automorphism of degree at least two. Then the set of isolated periodic points of $f^{\circ}$ is a set of bounded height in $\bba^N(\qb)$. 
\end{conj}

This conjecture was verified for affine plane automorphisms by Denis \cite{De95}, for regular automorphisms and for triangular automorphisms by Marcello \cite{Ma00, Ma03}. See also \cite{Le11, Le18, Si06} for some variants.

In this paper, motivated by this conjecture, we study the boundedness question about heights of periodic points of a cohomologically hyperbolic birational self-map (see Definition \ref{coh}). For a dominant rational self-map $f: X \dashrightarrow X$ of a smooth projective variety $X$, we let $I(f) \subset X$ be the indeterminacy locus of $f$ and let $X(\qb)_f$ be the set of $\qb$-rational points $P$ on $X$ such that $f^i(P) \notin I(f)$ for every $i \geq 0$. For $P \in X(\qb)_f$, its \textit{forward $f$-orbit} is defined as
\[ \OO_f(P) = \{f^i(P) \,:\, i \geq 0\}. \]
Assume that $f : X\dashrightarrow X$ is birational. Then for $P \in X(\qb)_f \cap X(\qb)_{f^{-1}}$, its \textit{full $f$-orbit} is defined as 
\[ \bo_f(P) = \{f^i(P) \,:\, i \in \ZZ \}.
\]
Let $h_H$ be a logarithmic Weil height associated with an ample divisor $H$ on $X$ (see \cite[Part B]{HS00}). We say a subset $W \subset X(\qb)$ is a set of bounded height, if there is a constant $c_H$ such that $h_H(Q) \leq c_H$ for all $Q \in W$.

\begin{theorem}\label{bdd} Let $f:  X \dashrightarrow X$ be a birational self-map of a smooth projective variety $X$. Assume that $f$ is cohomologically hyperbolic. Then there is a proper Zariski closed subset $Z \subsetneq X$, such that the set 
\[ \{ P \in X(\qb)_f \cap X(\qb)_{f^{-1}} : P \text{ is } f\text{-periodic and } \bo_f(P) \cap Z(\qb) = \varnothing \}
\] 
is a set of bounded height in $X(\qb)$.  
\end{theorem}

The proper Zariski closed subset $Z$ appeared in Theorem \ref{bdd} is usually not $f$-periodic unlike \cite[Theorem 6.4]{Ka08}. The proof given in Section \ref{thm2} is inspired by the study of arithmetic dynamics of regular affine space automorphisms which was achieved progressively by a series of works of Silverman, Denis, Marcello, Kawaguchi, and Lee (\cite{De95, Ka06, Ka13, Le13, Ma00, Ma03, Si94, Si06}). 

In the rest of this paper, we use a refinement of the argument to study a conjecture of Kawaguchi and Silverman concerning the properties of arithmetic degrees. For $P \in X(\qb)_f$, the \textit{arithmetic degree} of $f$ at $P$ is the quantity 
\[ \alpha_f(P) = \lim_{i\to +\infty} h^{+}_H(f^i(P))^{1/i}
\]
assuming that the limit exists, where $h^{+}_H = \max\{h_H, 1\}$. The Kawaguchi--Silverman conjecture predicts the existence of $\alpha_f(P)$ together with the equality $\alpha_f(P) = \delta_1(f)$ if in addition, $P$ has Zariski dense forward $f$-orbit $\OO_f(P)$ in $X$. Here $\delta_1(f)$ denotes the first dynamical degree of $f$. We refer to Section \ref{prel} for the precise statement of this conjecture and related notions.

In this paper, we focus on the case of certain rational self-maps in dimension two and reduce the problem to the following \textit{dynamical Mordell--Lang conjecture} (\textit{cf.} \cite[Conjecture 1.5.0.1 and Question 3.2.0.1]{BGT16}). See \cite{BGT16, BHS20} and references therein for some evidence supporting this conjecture.

\begin{conjecture}\label{dml} Let $X$ be a quasi-projective variety, $f : X \dashrightarrow X$ a dominant rational self-map, and $Y \subset X$ a $($closed$)$ subvariety of $X$. If $P \in X(\qb)_f$, then the set 
\[ \{ n \in \ZZ_{>0} : f^n(P) \in Y \} 
\]
is a union of finitely many arithmetic progressions. 
\end{conjecture}

\begin{theorem}\label{dt} Let $f: S \dashrightarrow S$ be a dominant rational self-map of a smooth projective surface $S$ satisfying $\delta_1(f) > \delta_2(f)$, where $\delta_2(f)$ denotes the second dynamical degree of $f$. Assume that the dynamical Mordell--Lang conjecture holds for $(S, f)$. Let $P \in S(\qb)_f$ with Zariski dense forward $f$-orbit $\OO_f(P)$ in $S$. Then $\alpha_f(P)$ exists and equals $\delta_1(f)$. 
\end{theorem}

Combined with solutions to the dynamical Mordell--Lang conjecture in certain cases (\cite{BGT10, Xi17}), we obtain the following result.

\begin{theorem}\label{aff} Let $f: S \dashrightarrow S$ be a dominant rational self-map of a smooth projective surface $S$ satisfying $\delta_1(f) > \delta_2(f)$. Assume that there exists a non-empty open subset $W\subset S$ such that, either 

\medskip $(1)$ $f|_W: W \to W$ is an \'{e}tale surjective morphism; or 
	
\medskip $(2)$ $W = \bba^2$, and $f|_{\bba^2}: \bba^2 \to \bba^2$ is a surjective morphism. 
	
\medskip \noindent Let $P \in S(\qb)_f$ with Zariski dense forward $f$-orbit $\OO_f(P)$ in $S$. Then $\alpha_f(P)$ exists and equals $\delta_1(f)$. 
\end{theorem}

Theorem \ref{aff} solves the Kawaguchi--Silverman conjecture (Conjecture \ref{ksc}(d)) for certain endomorphisms of smooth quasi-projective surfaces, including all automorphisms. 

\begin{remark}\label{jsxz} $(1)$ We note that Conjecture \ref{ksc}(d) was verified for automorphisms (\cite{Ka08}, see also \cite[Theorem 2(c)]{KS14}) and for endomorphisms (\cite[Theorem 1.3]{MSS18}) of smooth projective surfaces. It was generalized to endomorphisms of arbitrary projective surfaces in \cite[Theorem 1.3]{MZ19}. In \cite{JSXZ21}, the authors showed that Conjecture \ref{ksc}(d) holds for endomorphisms of some smooth affine surfaces (see \cite[Theorems 1.8(2), 1.9(2) and 1.9(3)]{JSXZ21}), while the case of endomorphisms of $\bba^2$ was not included. 

Let us also point out that, in order to generalize these results to arbitrary rational self-maps of smooth projective surfaces, the essential difficulty is again that the theory of height functions does not behave functorially under rational maps.

\medskip $(2)$ For endomorphisms of $\bba^2$, there is a related work by Jonsson and Wulcan \cite{JW12}. In particular, they proved Conjecture \ref{ksc}(d) for automorphisms of $\bba^2$ (see Remark \ref{ex-jw}(1)). The referee also pointed out an unpublished work of Jonsson, Wulcan and Xie. The authors showed Conjecture \ref{ksc}(d) for a large class of endomorphisms of $\bba^2$ (including all endomorphisms $f$ satisfying $\delta_1(f) \geq \delta_2(f)$), which covered Theorem \ref{aff}(2). Their argument is based on the theory of the valuative tree developed by Favre and Jonsson, which is very different from ours. See \cite{Xi17c} for a summary. 
\end{remark}

As a byproduct of the argument of Theorem \ref{dt}, we show the following existence result about Zariski dense orbits, which can be regarded as a supplement to Theorem \ref{dt}.

\begin{theorem}\label{dt2} Let $f: S \dashrightarrow S$ be a dominant rational self-map of a smooth projective surface $S$ satisfying $\delta_1(f) > \delta_2(f)$. 

\medskip $(1)$ Then there exists a $\qb$-point $P \in S(\qb)_f$ with Zariski dense forward $f$-orbit $\OO_f(P)$ in $S$. 

\medskip $(2)$ If in addition, there is a non-empty open subset $W\subset S$ such that $f|_W: W \to W$ is a surjective morphism, then there exists a $\qb$-point $P \in W(\qb)$ with Zariski dense orbit $\OO_f(P)$ in $W$.
\end{theorem}



In general, the so-called Zariski dense orbit conjecture predicts that there exists an algebraic point with Zariski dense forward orbit, provided that every $f$-invariant rational function is constant. This conjecture was verified for endomorphisms of $\bba^2$ (\cite{Xi17b}), for birational self-maps or endomorphisms of smooth projective surfaces (\cite{Xi15, Xi22}, see also \cite{JXZ20}), and for endomorphisms of some smooth affine surfaces (\cite{JSXZ21}). See \cite[Introduction]{Xi22} for more information about this conjecture.

We will prove Theorems \ref{dt} and \ref{dt2} (together with Theorem \ref{aff}) in Section \ref{thm1} after some preliminaries about dynamical and arithmetic degrees in Section \ref{prel} and technical lemmas in Section \ref{tel}. 

\subsection*{Acknowledgements.} The author thanks Professor Keiji Oguiso for valuable suggestions and constant support. He also thanks Professors Mattias Jonsson, Shu Kawaguchi, Yohsuke Matsuzawa, Junyi Xie, De-Qi Zhang, and the referee for beneficial discussions and comments. He is grateful to Shanghai Center for Mathematical Sciences and Professor Chen Jiang for warm hospitality, and to Professors Meng Chen and Quanshui Wu for their encouragement. This work is supported by JSPS KAKENHI Grant Number 21J10242.


\section{Preliminaries}\label{prel}

Let $f : X \dashrightarrow X$ be a dominant rational self-map of a smooth projective variety $X$. Given an ample divisor $H$ on $X$, and any integer $0\leq k \leq \dim X$, one defines the $k$-th degree of $f$ as the integer \[ \deg_{k, H} (f) = \big( f^{\ast}(H^k)\cdot H^{\dim X - k} \big). 
\]
Then the \textit{$k$-th dynamical degree} of $f$ is defined as
\[  \delta_k(f) = \lim_{i \to +\infty} \left( \deg_{k, H} (f^i) \right)^{1/i}. 
\]
This limit exists and is independent of the choice of ample divisor $H$ (\cite{DS04, DS05, Tr20, Da20}). In fact, by \cite[Theorem 5.3.2]{Da20}, 
\[ \delta_k(f) = \lim_{i \to +\infty} \big\Vert (f^i)^{\ast}|_{N^k(X)_{\RR}} \big\Vert^{1/i}, \]
where $N^k(X)$ denotes the group of numerical classes of cycles of codimension $k$ on $X$, $N^k(X)_{\RR} = N^k(X) \otimes \RR$, $(f^i)^{\ast} : N^k(X)_{\RR} \to N^k(X)_{\RR}$ is induced by $f^i$, and 
$\Vert \cdot \Vert$ denotes some fixed operator norm. 
In particular, when $f$ is a surjective endomorphism, or more generally, $f$ is \textit{algebraically} $k$-\textit{stable}, that is, $(f^i)^{\ast} = (f^{\ast})^i$ on $N^k(X)_{\RR}$ for all $i \geq 1$, we have
\[ \delta_k(f) = \text{spectral radius of}\ f^{\ast}: N^k(X)_{\RR} \to N^k(X)_{\RR}. \]
It is known that the function $k \mapsto \log \delta_k(f)$ is concave (see e.g., \cite[Proposition 1.2(i)]{Gu05}). Equivalently,  we have
\[ \delta_k(f)^2 \geq \delta_{k-1}(f) \delta_{k+1}(f) \ \ \text{for } 1 \leq k \leq \dim X - 1. 
\]
In particular, we have $1 \leq \delta_k(f) \leq \delta_1(f)^k$, and there are positive integers $p$ and $q$ with $0 \leq p \leq q \leq \dim X$ such that
\begin{equation}\label{kt}
1 = \delta_0(f) < \cdots < \delta_p(f) = \cdots = \delta_q(f) > \cdots > \delta_{\dim X}(f). 
\end{equation}

\begin{definition}\label{coh} Let $f : X \dashrightarrow X$ be a dominant rational self-map. We say $f$ is \textit{cohomologically hyperbolic}, 
if there is a positive integer $p \in \{1, 2, \dots, \dim X\}$ such that the $p$-th dynamical degree
\[ \delta_p(f) > \delta_i(f)  \ \ \text{for all } p \neq i \in \{0, 1, \dots, \dim X\}. 
\]
(or equivalently, for both $i = p \pm 1$ due to (\ref{kt}) where we set $\delta_{\dim X+1}(f) = 0$). 
\end{definition}

Notice that the cohomological hyperbolicity of $f$ implies that $\delta_1(f) > 1$. Moreover, for a birational self-map $f$, it is cohomologically hyperbolic, if and only if $\delta_1(f) > 1$ when $\dim X = 2$; if and only if $\delta_1(f) \neq \delta_2(f)$ when $\dim X = 3$.

Recall that for a dominant rational self-map $f: X \dashrightarrow X$, we
let $I(f) \subset X$ be the indeterminacy locus of $f$ and let $X(\qb)_f$ be the set of $\qb$-rational points $P$ on $X$ such that $f^i(P) \notin I(f)$ for every $i \geq 0$. By a result of Amerik \cite{Am11}, the set $X(\qb)_f$ is Zariski dense in $X$. For $P \in X(\qb)_f$, its forward $f$-orbit is defined as $\OO_f(P) = \{f^i(P) \,:\, i \geq 0\}$. If $f : X\dashrightarrow X$ is a birational self-map, then for $P \in X(\qb)_f \cap X(\qb)_{f^{-1}}$, its full $f$-orbit is defined as $\bo_f(P) = \{f^i(P) \,:\, i \in \ZZ \}$. Clearly $\bo_f(P) = \OO_f(P) \cup \OO_{f^{-1}}(P)$, and moreover, $P$ is $f$-periodic if and only if $\bo_f(P)$ is finite.

Given an ample divisor $H$ on a smooth projective variety $X$, one can construct a corresponding \textit{logarithmic Weil height} $h_H : X(\qb) \to  \RR$ that measures the arithmetic complexity of the point $P$ (see \cite[Part B]{HS00}). For example, if $X = \PP^N$, $H = \mathcal{O}_{\PP^N}(1)$, and $P = [a_0 : \cdots : a_N]$ is a $\QQ$-point with the $a_j$ coprime integers, then $h_H(P) = \log\left(\max\{\vert a_j\vert\}\right)$. Let $h_H$ be a logarithmic Weil height associated with an ample divisor $H$ on $X$, and $h^{+}_H = \max\{h_H, 1\}$. Following \cite{KS16a, Si14}, for $P \in X(\qb)_f$, we define the \textit{lower} and the \textit{upper arithmetic degree} of $f$ at $P$ by 
\[ \underline{\alpha}_f(P) = \liminf_{i\to +\infty} h^{+}_H(f^i(P))^{1/i}, \]
\[ \overline{\alpha}_f(P) = \limsup_{i\to +\infty} h^{+}_H(f^i(P))^{1/i}. \]
Both of these quantities are again independent of the choice of ample divisor $H$ (\cite[Proposition 12]{KS16a}). When $\underline{\alpha}_f(P) = \overline{\alpha}_f(P)$, we write
\[ \alpha_f(P) = \lim_{i\to +\infty} h^{+}_H(f^i(P))^{1/i}. 
\]

The following is a conjecture of Kawaguchi and Silverman concerning the properties of arithmetic degrees. 

\begin{conjecture}[{\cite[Conjecture 6]{KS16a}, see also \cite[Conjecture 40]{Si14}}]\label{ksc} Let $f : X \dashrightarrow X$ be a dominant rational self-map of a smooth projective variety $X$ and $P \in X(\qb)_f$. Then the following hold.
	
\medskip $(\mathrm{a})$ $\underline{\alpha}_f(P) = \overline{\alpha}_f(P)$. In other words, the limit 
\[ \lim_{i\to +\infty} h^{+}_H(f^i(P))^{1/i} \]
defining $\alpha_f(P)$ exists.
	
\medskip $(\mathrm{b})$ $\alpha_f(P)$ is an algebraic integer.
	
\medskip $(\mathrm{c})$ The collection of arithmetic degrees
\[ \{ \alpha_f(Q) \,:\, Q\in X(\qb)_f \} \]
is a finite set.
	
\medskip $(\mathrm{d})$ If the forward orbit $\OO_f(P)$ is Zariski dense in $X$, then
\[ \alpha_f(P) = \delta_1(f). 
\] 
\end{conjecture}

In general, we have 
\[ 1 \leq \underline{\alpha}_f(P) \leq \overline{\alpha}_f(P) \leq \delta_1(f), 
\]
where the first two inequalities are obvious, while the last one is non-trivial and is due to \cite[Theorem 1.4]{Ma20}. In particular, if $\delta_1(f) = 1$, then $\alpha_f(P) = 1 = \delta_1(f)$ for all $P \in X(\qb)_f$. Therefore, we are only interested in the case where $\delta_1(f) > 1$. Moreover, in order to show the existence of $\alpha_f(P)$ and the equality $\alpha_f(P) = \delta_1(f)$, it is equivalent to show the inequality $\underline{\alpha}_f(P) \geq \delta_1(f)$.

When $f: X \to X$ is a surjective endomorphism, the first three claims (a), (b) and (c) were shown in \cite{KS16b}, while the last one (d) is still widely open. We refer to \cite{CLO22, MSS18, MMSZ20} and references therein for known results.

When $f: X \dashrightarrow X$ is an arbitrary dominant rational self-map, much less is known. In fact, 
there is a counterexample to Conjecture \ref{ksc}(c) due to \cite{LS20}. Furthermore, the results from \cite{BDJ20, BDJK21} suggest that claim (b) may also be false. Let us list some previous works towards claim (d). When $f$ is a birational self-map of an irrational smooth projective surface, one can reduce the problem to the induced automorphism of the minimal model (\cite[Theorem 3.6]{MSS18}). While for a birational self-map of a projective space, the question is much sophisticated; Conjecture \ref{ksc}(d) is still open even in dimension two. The main obstacle is the deficiency of functoriality of height functions (\cite[Theorem B.3.2.(b)]{HS00}) for an arbitrary rational map. Known results include regular affine space automorphisms (\cite{Ka06, Ka13, Le13}, see also \cite[Theorem 2(b)]{KS14}) and monomial maps (\cite{Si14}, see also \cite[Theorem 2(d)]{KS14} and \cite{Li19}).

We collect some basic facts here about arithmetic degrees for later use.

\begin{lemma}\label{ks13} Let $f : X \dashrightarrow X$ be a dominant rational self-map of a smooth projective variety $X$. Then for all $P\in X(\qb)_f$ and all $k \geq 0$, 
\[ \underline{\alpha}_f(f^k(P)) = \underline{\alpha}_f(P) \ \ \text{and} \ \ \overline{\alpha}_f(f^k(P)) = \overline{\alpha}_f(P). 
\]
\end{lemma}

\begin{proof} This is clear by definition (see e.g., \cite[Lemma 13]{KS16a}). 
\end{proof}

\begin{lemma}\label{est} Let $f : X \dashrightarrow X$ be a dominant rational self-map of a smooth projective variety $X$. Let $P\in X(\qb)_f$ and $n\geq 1$. Then 
\begin{equation*}
\underline{\alpha}_{f^n}(P) = \underline{\alpha}_{f}(P)^n \ \ \text{and} \ \ \overline{\alpha}_{f^n}(P) = \overline{\alpha}_{f}(P)^n. 
\end{equation*} 
\end{lemma}

The above equalities are obvious if we know the existence of $\alpha_f(P)$, for instance, when $f$ is a morphism.

\begin{proof} These two equalities are stated in \cite[3.9(iii)]{JSXZ21}. The second one is exactly \cite[Corollary 3.4]{Ma20}; see also the proof of \cite[Proposition 12]{Si14}. Here we only show the first one by an analogous argument. By definition, 
\begin{equation*}
\begin{aligned}
\underline{\alpha}_{f^n}(P)^{1/n} = \liminf_{i\to +\infty} h^{+}((f^n)^i(P))^{1/ni} 
\geq \liminf_{j\to +\infty} h^{+}(f^{j}(P))^{1/j} = \underline{\alpha}_{f}(P).  
\end{aligned}
\end{equation*} 
This proves $\underline{\alpha}_{f^n}(P) \geq \underline{\alpha}_{f}(P)^n$. Write $j = qn+r$, where $q \geq 0$ and $n-1 \geq r \geq 0$. Then by \cite[Lemma 3.3]{Ma20}, 
\[ h^{+}(f^{qn+n}(P)) \leq C_0^{n-r} h^{+}(f^{j}(P)) \leq C_0^{n} h^{+}(f^{j}(P)),  
\]
where the constant $C_0 \geq 1$ is independent of the choice of $P$, $n$ and $j$. So 
\begin{equation*}
\begin{aligned}
&h^{+}(f^{qn+n}(P))^{1/(qn+n)} 
\leq C_0^{1/(q+1)} h^{+}(f^{j}(P))^{1/j}. 
\end{aligned}
\end{equation*}
It then follows that 
\begin{equation*}
\begin{aligned}
\underline{\alpha}_{f}(P) &= \liminf_{q\to +\infty} \min_{0\leq r \leq n-1} h^{+}(f^{qn+r}(P))^{1/(qn+r)} \\ 
&\geq \liminf_{q\to +\infty} \min_{0\leq r \leq n-1} C_0^{-1/(q+1)}\cdot h^{+}(f^{qn+n}(P))^{1/(qn+n)} \\ 
&= \liminf_{q\to +\infty} h^{+}(f^{qn+n}(P))^{1/(qn+n)} = \underline{\alpha}_{f^n}(P)^{1/n}. 
\end{aligned}
\end{equation*} 
This proves  $\underline{\alpha}_{f}(P)^n \geq \underline{\alpha}_{f^n}(P)$. 
\end{proof}

The condition of Zariski dense orbit in Conjecture \ref{ksc}(d) is related with the dynamical Mordell--Lang conjecture (Conjecture \ref{dml}). This is indicated by the following easy observation.

\begin{lemma}[{\textit{cf.} \cite[Corollary 1.4]{BGT10}}]\label{obs} Let $f : X \dashrightarrow X$ be a dominant rational self-map of a smooth projective variety $X$ and $P\in X(\qb)_f$ with Zariski dense orbit $\OO_f(P)$ in $X$. Let $U \subset X$ be a non-empty Zariski open subset. Assume Conjecture \ref{dml} holds for $(X, f)$. Then the set 
\[ \{ n \in \ZZ_{>0} : f^n(P) \notin U \} 
\]
is finite. 
\end{lemma}

\begin{proof} The complement $X \setminus U$ is a finite union of (irreducible) closed subvarieties of $X$. We claim that for every subvariety $Y \subsetneq X$, the set 
\[ \{ n \in \ZZ_{>0} : f^n(P) \in Y \} \]
is finite. Otherwise, $(f^a)^k(f^b(P)) = f^{ak+b}(P)\in Y$ for some $a, b \in \ZZ_{>0}$ and all $k\geq 1$. This is absurd because $\OO_f(P)$ is Zariski dense in $X$. 
\end{proof}

We state the following two cases for which Conjecture \ref{dml} holds, and refer to \cite{BGT16, BHS20}  and references therein for other known results. 

\begin{theorem}[{\cite[Theorem 1.3]{BGT10}}]\label{bgt1.3} Conjecture \ref{dml} holds for all \'{e}tale morphisms $f: X \to X$ with $X$ quasi-projective. 
\end{theorem}

\begin{theorem}[{\cite[Theorem 1]{Xi17}}]\label{xia} Conjecture \ref{dml} holds for all polynomial endomorphisms $f : \bba^2 \to \bba^2$. 
\end{theorem}

We end this section with the following two facts which will be used later. 

\begin{lemma}[{\cite[Theorem 4.3 (i) and (ii)]{Xi14}}]\label{xi4.3} Let $S$ be a smooth projective surface and $f : S \dashrightarrow S$ a dominant rational map. 

\medskip $(1)$ For any positive integer $m \geq 1$, Conjecture \ref{dml} holds for $(S, f)$ if and only if it holds for $(S, f^m)$. 

\medskip $(2)$ Let $U \subset S$ be a non-empty open subset of $S$ such that the restriction $f|_U : U \to U$ is a morphism.  Then Conjecture \ref{dml} holds for $(S, f)$ if and only if it holds for $(U, f|_U)$. 
\end{lemma}


\section{Technical lemmas}\label{tel}

In order to show Theorems \ref{bdd} and \ref{dt}, we need the following two technical results (Lemmas \ref{tec} and \ref{tech}), whose proofs are refinements of the argument of \cite[Theorem 7.19]{Si07}. 

\begin{lemma}\label{tec} Let $g: X \dashrightarrow X$ be a dominant rational self-map of a smooth projective variety $X$. Fix an ample divisor $H$ on $X$. Let 
	\[ d_1, d_2 > 0, \ \ \text{and} \ \ \zeta > \frac{1}{d_1} + \frac{1}{d_2} = \frac{d_1 + d_2}{d_1d_2}  
	\]
	be constants. Assume that 
	\[ \frac{h_H(g^2(R))}{d_1} + \frac{h_H(R)}{d_2} \geq \zeta h_H(g(R)) - C
	\]
	holds for all $R \in X(\qb)_g \cap U(\qb)$, where $U \subset X$ is a non-empty open subset, and $C > 0$ is independent of the choice of $R$. 
	
	Let $Q\in X(\qb)_g$ be a $\qb$-point such that its forward $g$-orbit $\OO_g(Q) \subset U(\qb)$. If $\OO_g(Q)$ is infinite, then 
	\[ \underline{\alpha}_g(Q) \geq \frac{\zeta d_1}{2}. 
	\] 
\end{lemma}

\begin{proof} We write $h = h_H$ as $H$ is fixed. Put 
	\[ Q_k = g^k(Q)
	\]
	for all $k\geq 0$, where $Q_0 = Q$. Then $g^n(Q_k) \in X(\qb)_g \cap U(\qb)$ for all $n, k \geq 0$, and hence 
	\[ \frac{h(g^{n+2}(Q_k))}{d_1} + \frac{h(g^n(Q_k))}{d_2} \geq \zeta h(g^{n+1}(Q_k)) - C
	\]
	We define the sequence 
	\begin{equation*}
		\ell_n(Q_k) = \frac{h(g^{n+1}(Q_k))}{d_1} - \frac{h(g^{n}(Q_k))}{\alpha d_2} - \frac{C}{\alpha-1}
	\end{equation*} 
	for all $n, k \geq 0$, where 
	\begin{equation}\label{defa}
		\alpha = \frac{\zeta d_1d_2 + \sqrt{\zeta^2 d_1^2d_2^2 - 4d_1d_2}}{2d_2}.  
	\end{equation}
	We note that 
	\begin{equation}\label{esta} 
		\alpha > \frac{d_1 + d_2 + \vert d_1 - d_2 \vert}{2d_2} \geq 1, 
	\end{equation}
	\begin{equation}\label{esta2} 
		\alpha \geq \frac{\zeta d_1d_2}{2d_2} = \frac{\zeta d_1}{2}, 
	\end{equation}
	and 
	\begin{equation}
		\frac{\alpha}{d_1} + \frac{1}{\alpha d_2} = \zeta. 
	\end{equation} 
	Then $\ell_n(Q_k)$ satisfies that
	\begin{equation}\label{cal}
		\begin{aligned}
			&\ell_{n+1}(Q_k) - \alpha \ell_n(Q_k) \\ 
			= \ &\left( \frac{h(g^{n+2}(Q_k))}{d_1} - \frac{h(g^{n+1}(Q_k))}{\alpha d_2} - \frac{C}{\alpha-1} \right) \\
			&- \alpha \left( \frac{h(g^{n+1}(Q_k))}{d_1} - \frac{h(g^{n}(Q_k))}{\alpha d_2} - \frac{C}{\alpha-1} \right) \\
			= \ &\left( \frac{h(g^{n+2}(Q_k))}{d_1} + \frac{h(g^{n}(Q_k))}{d_2} \right) - \left( \frac{\alpha}{d_1} + \frac{1}{\alpha d_2} \right) h(g^{n+1}(Q_k)) + C \\ 
			\geq \ &\left( \zeta - \frac{\alpha}{d_1} -  \frac{1}{\alpha d_2} \right) h(g^{n+1}(Q_k)) = 0.  
		\end{aligned} 
	\end{equation}
	Thus, by induction on $n$, we obtain that
	\begin{equation*}
		\ell_n(Q_k) \geq \alpha^n\ell_0(Q_k)
	\end{equation*}
	for all $n, k \geq 0$, which is equivalent to 
	\begin{equation}\label{key}
		\begin{aligned}
			&\frac{h(g^{n+1}(Q_k))}{\alpha^n d_1} - \frac{C}{\alpha^n (\alpha-1)} \\
			\geq \  &\frac{h(g^{n}(Q_k))}{\alpha^{n+1} d_2} + \frac{h(g(Q_k))}{d_1} - \frac{h(Q_k)}{\alpha d_2} - \frac{C}{\alpha-1}
		\end{aligned}
	\end{equation}
	by the definition of $\ell_n(Q_k)$.

	We define 
	\begin{equation*}
		\ell_{\infty}(Q_k) = \liminf_{i\to +\infty} \frac{h(g^{i}(Q_k))}{\alpha^i}.  
	\end{equation*}
	Then clearly 
	\begin{equation*}
		0\leq \ell_{\infty}(Q_k) \leq +\infty, 
	\end{equation*}
	and 
	\begin{equation}\label{ite} 
		\ell_{\infty}(Q_k) = \alpha^k \ell_{\infty}(Q) \ \ \text{for all } k\geq 0. 
	\end{equation} 
	Letting $n\to +\infty$ and taking liminf in (\ref{key}), we obtain that
	\begin{equation}\label{key2}
		\frac{\alpha \ell_{\infty}(Q_k)}{d_1} \geq \frac{\ell_{\infty}(Q_k)}{\alpha d_2} + \frac{h(g(Q_k))}{d_1} - \frac{h(Q_k)}{\alpha d_2} - \frac{C}{\alpha-1}.   
	\end{equation}

	By assumption, $\OO_g(Q)$ is infinite. We claim that $\ell_{\infty}(Q) > 0$. Otherwise, $\ell_{\infty}(Q) = 0$, and hence $\ell_{\infty}(Q_k) = 0$ for all $k\geq 0$ due to (\ref{ite}). It then follows that 
	\begin{equation}\label{iteb}
		\frac{C}{\alpha-1} \geq \frac{h(g(Q_k))}{d_1} - \frac{h(Q_k)}{\alpha d_2}. 
	\end{equation}
	from (\ref{key2}). We set
	\begin{equation}\label{deb}
		\beta  = \frac{\zeta d_1d_2 + \sqrt{\zeta^2 d_1^2d_2^2 - 4d_1d_2}}{2d_1}, 
	\end{equation}
	which clearly satisfies 
	\begin{equation}
		\alpha d_2 = \beta d_1, \ \ \beta > 1, \ \ \text{and} \ \ \frac{\beta}{d_2} + \frac{1}{\beta d_1} = \zeta. 
	\end{equation}
	Thus, we have the inequality 
	\begin{equation*}
		\frac{d_1 C}{\alpha-1} \geq h(g(Q_k)) -  \frac{h(Q_k)}{\beta}
	\end{equation*} 
	from (\ref{iteb}) for all $k \geq 0$, which by the definition of $Q_k$ is equivalent to 
	\begin{equation} 
		\frac{d_1 C}{\alpha-1} \geq h(g^{k+1}(Q)) - \frac{h(g^{k}(Q))}{\beta}. 
	\end{equation}
	Therefore, for $n \geq 1$, 
	\begin{equation*}
		\begin{aligned}
			h(g^n(Q)) &= \sum^n_{i=1} \frac{1}{\beta^{n-i}} \left( h(g^i(Q)) - \frac{h(g^{i-1}(Q))}{\beta} \right) + \frac{h(Q)}{\beta^n} \\ 
			&\leq \left( \sum^n_{i=1} \frac{1}{\beta^{n-i}} \right) \frac{d_1 C}{\alpha-1} + \frac{h(Q)}{\beta^n} \\
			&\leq \dfrac{\beta d_1 C}{(\beta - 1)(\alpha - 1)} + \vert h(Q) \vert, 
		\end{aligned}
	\end{equation*}
	where we use the facts that $\alpha, \beta > 1$, and $C > 0$. This implies that the set $\OO_g(Q) = \{ g^n(Q) : n\geq 0 \}$ has bounded height. It is clear that $\OO_g(Q)$ also has bounded degree over $\QQ$. Hence by the Northcott finiteness property (\cite[Theorem B.3.2.(g)]{HS00}), we obtain that $\OO_g(Q)$ is a finite set. This is a contradiction. Therefore, $\ell_{\infty}(Q) > 0$.

	Since 
	\[ \ell_{\infty}(Q) = \liminf_{i\to +\infty} \frac{h(g^{i}(Q))}{\alpha^i} 
	\] 
	by definition, we can easily conclude that 
	\[ \liminf_{i\to +\infty} h(g^i(Q))^{1/i} \geq \alpha. 
	\]
	It then follows that 
	\begin{equation*}
		\begin{aligned}
			\underline{\alpha}_g(Q) &= \liminf_{i\to +\infty} h^{+}(g^i(Q))^{1/i} \\
			&\geq \liminf_{i\to +\infty} h(g^i(Q))^{1/i} \\ 
			&\geq \alpha \geq \frac{\zeta d_1}{2}, 
		\end{aligned}
	\end{equation*}
	where the last inequality follows from (\ref{esta2}). 
\end{proof}

\begin{lemma}\label{tech} Let $f: X \dashrightarrow X$ be a birational self-map of a smooth projective variety $X$. Fix an ample divisor $H$ on $X$. Let 
\[ d_1, d_2 > 0, \ \ \text{and} \ \ \zeta > \frac{1}{d_1} + \frac{1}{d_2} = \frac{d_1 + d_2}{d_1d_2}  
\]
be constants. Assume that we have
\begin{equation}\label{hiq}
\frac{h_H(f(R))}{d_1} + \frac{h_H(f^{-1}(R))}{d_2} \geq \zeta h_H(R) - C
\end{equation}
for all $R\in X(\qb)_f \cap X(\qb)_{f^{-1}} \cap U(\qb)$ where $U\subset X$ is a non-empty Zariski open subset, and the constant $C > 0$ is independent of the choice of $R$. 

Then there is a constant $\widetilde{C} > 0$ such that 
\[ h_H(Q) \leq \widetilde{C} 
\]
for all $Q\in X(\qb)_f \cap X(\qb)_{f^{-1}}$ satisfying $\bo_f(Q) \subset U(\qb)$ and $\bo_f(Q)$ is finite. 
\end{lemma}

\begin{proof} We write $h = h_H$ as $H$ is fixed. Let $Q\in X(\qb)_f \cap X(\qb)_{f^{-1}}$ be a $\qb$-point with $\bo_f(Q) \subset U(\qb)$. Then $f^n(Q) \in X(\qb)_f \cap X(\qb)_{f^{-1}} \cap U(\qb)$ for all $n \geq 0$, and hence 
\[ \frac{h(f^{n+1}(Q))}{d_1} + \frac{h(f^{n-1}(Q))}{d_2} \geq \zeta h(f^n(Q)) - C. 
\]
Define the sequence 
\begin{equation*}
\ell^{+}_n(Q) = \frac{h(f^n(Q))}{d_1} - \frac{h(f^{n-1}(Q))}{\alpha d_2} - \frac{C}{\alpha-1}
\end{equation*} 
for all $n \geq 0$, where 
\begin{equation}\label{defa2}
\alpha = \frac{\zeta d_1d_2 + \sqrt{\zeta^2 d_1^2d_2^2 - 4d_1d_2}}{2d_2}. 
\end{equation}
We note that 
\begin{equation}\label{esta22} 
\alpha > 1, \text{ and } \frac{\alpha}{d_1} + \frac{1}{\alpha d_2} = \zeta. 
\end{equation}
Then similar to (\ref{cal}), $\ell^{+}_n(Q)$ satisfies that
\begin{equation*}
\ell^{+}_{n+1}(Q) \geq \alpha \ell^{+}_n(Q), 
\end{equation*}
and hence by induction on $n$, we obtain that
\begin{equation*}
\ell^{+}_{n+1}(Q) \geq \alpha^n \ell^{+}_1(Q)
\end{equation*}
for all $n \geq 0$.

On the other hand, as $f^{-n}(Q) \in X(\qb)_f \cap X(\qb)_{f^{-1}} \cap U(\qb)$ for all $n \geq 0$, we have 
\begin{equation*}\label{estf}
\frac{h(f^{-n+1}(Q))}{d_1} + \frac{h(f^{-n-1}(Q))}{d_2} \geq \zeta h(f^{-n}(Q)) - C.
\end{equation*}
Define the sequence 
\begin{equation*}
\ell^{-}_n(Q) = \frac{h(f^{-n}(Q))}{d_2} - \frac{h(f^{-n+1}(Q))}{\beta d_1} - \frac{C}{\beta-1}
\end{equation*} 
for all $n \geq 0$, where 
\begin{equation*}
\beta  = \frac{\zeta d_1d_2 + \sqrt{\zeta^2 d_1^2d_2^2 - 4d_1d_2}}{2d_1}. 
\end{equation*}
Clearly 
\begin{equation*}
\alpha d_2 = \beta d_1, \beta > 1, \text{ and }	\frac{\beta}{d_2} + \frac{1}{\beta d_1} = \zeta. 
\end{equation*}
An analogous calculation to (\ref{cal}) shows that 
\begin{equation*}
\ell^{-}_{n+1}(Q) \geq \beta \ell^{-}_n(Q), 
\end{equation*} 
from which we deduce that
\begin{equation*} 
\ell^{-}_{n+1}(Q) \geq \beta^n\ell^{-}_{1}(Q)
\end{equation*}
for all $n \geq 0$. It follows that 
\begin{equation*}
\begin{aligned}
&\frac{\ell^{+}_{n+1}(Q)}{\alpha^n} + \frac{\ell^{-}_{n+1}(Q)}{\beta^n} \geq \ell^{+}_1(Q) + \ell^{-}_1(Q) \\
= \ &\frac{h(f(Q))}{d_1} - \frac{h(Q)}{\alpha d_2} - \frac{C}{\alpha-1} + \frac{h(f^{-1}(Q))}{d_2} - \frac{h(Q)}{\beta d_1} - \frac{C}{\beta-1} \\ 
\geq \ &\left( \zeta - \frac{1}{\alpha d_2} - \frac{1}{\beta d_1} \right) h(Q) - \left( 1 + \frac{1}{\alpha-1} + \frac{1}{\beta-1} \right) C, 
\end{aligned} 
\end{equation*}
This is equivalent to 
\begin{equation}\label{key3}
\begin{aligned} 
&\frac{h(f^{n+1}(Q))}{\alpha^n d_1} + \frac{h(f^{-n-1}(Q))}{\beta^n d_2} \\
&- \frac{C}{\alpha^n(\alpha-1)} -  \frac{C}{\beta^n(\beta-1)} + \frac{(\alpha\beta-1)C}{(\alpha-1)(\beta-1)} \\
\geq \ &\left( \zeta - \frac{1}{\alpha d_2} - \frac{1}{\beta d_1} \right) h(Q) + \frac{h(f^{n}(Q))}{\alpha^{n+1} d_2} + \frac{h(f^{-n}(Q))}{\beta^{n+1} d_1}  
\end{aligned}
\end{equation}
by the definitions of $\ell^{+}$ and $\ell^{-}$. 

Assume that $\bo_f(Q)$ is finite. Then $h(f^{m}(Q))$ is bounded independently of $m \in \ZZ$. Letting $n \to +\infty$ in (\ref{key3}), we obtain that 
\begin{equation*}
\frac{(\alpha\beta-1)C}{(\alpha-1)(\beta-1)} \geq \left( \zeta - \frac{1}{\alpha d_2} - \frac{1}{\beta d_1} \right) h(Q). 
\end{equation*}
Note that 
\begin{equation*}
\zeta - \frac{1}{\alpha d_2} - \frac{1}{\beta d_1} = \frac{\sqrt{\zeta^2 d_1^2 d_2^2 - 4d_1d_2}}{d_1 d_2} > 0. 
\end{equation*}
Therefore, $h(Q) \leq \widetilde{C}$, where 
\begin{equation*}
\widetilde{C} = \frac{(\alpha\beta-1)C}{(\alpha-1)(\beta-1)} \cdot \left( \zeta - \frac{1}{\alpha d_2} - \frac{1}{\beta d_1} \right)^{-1} 
\end{equation*}
is independent of the choice of $Q$. 
\end{proof}


\section{Proof of Theorem \ref{bdd}}\label{thm2} 

Let us start to prove Theorem \ref{bdd}. Recall that a subset $W \subset X(\qb)$ is said to be \textit{of bounded height}, if there is a constant $c_H$ such that $h_H(P) \leq c_H$ for all $P \in W$. Notice that the definition of a set of bounded height is independent of the choice of ample divisors $H$ (see e.g., \cite[Theorem 4.1(vii)]{Ka08}). The Northcott finiteness property (\cite[Theorem B.3.2(g)]{HS00}) claims that if $W$ is a set of bounded height, then the set $\{ P \in W : [\QQ(P) : \QQ] \leq B \}$ is finite for each positive integer $B$.

In view of Lemma \ref{tech}, we only need to establish (\ref{hiq}), which will be done in the next lemma. The proof uses resolution of indeterminacy as in \cite{Si94, Ka06, Le13} and Siu's numerical criterion for bigness as in \cite{Si11} together with a new ingredient, namely, Dang's inequality (\cite[Corollary 3.4.5]{Da20}).

\begin{lemma}\label{aux} Let $f:  X \dashrightarrow X$ be a birational self-map of a smooth projective variety $X$. Fix an ample divisor $H$ on $X$. Assume that $f$ is cohomologically hyperbolic. Then there exist $n\in \ZZ_{>0}$, $A > 2$, $C > 0$ and a non-empty Zariski open subset $U\subset X$, such that 
\[ h_H(f^n(R)) + h_H(f^{-n}(R)) \geq Ah_H(R) - C \] 
for all $R\in X(\qb)_{f^n} \cap X(\qb)_{f^{-n}} \cap U(\qb)$. 
\end{lemma}

\begin{proof} Let $N = \dim X$. By assumption, there exists a positive integer $p \in \{1, 2, \dots, N-1\}$ such that the $p$-th dynamical degree
\[ \delta_p(f) > \delta_i(f)  \ \ \text{for all } p \neq i \in \{0, 1, \dots, N\}. 
\]
Then clearly 
\[ \delta_p(f)^{2} > \delta_{i}(f) \delta_{i+1}(f) \ \ \text{for all } i \in \{0, 1, \dots, N-1\}. 
\]
Since 
\[ \lim_{m \to +\infty} \left( \frac{\deg_p(f^{2m})}{\deg_i(f^m) \deg_{i+1}(f^m)} \right)^{1/m} = \frac{\delta_p(f)^2}{\delta_i(f) \delta_{i+1}(f)} > 1, 
\]
where 
\[ \deg_i (f) = \deg_{i, H} (f) = \left( f^{\ast}(H^{i})\cdot H^{N-i} \right), 
\]
we have some $n\in \ZZ_{>0}$ such that 
\[ \frac{\deg_p(f^{2n})}{\deg_i(f^n) \deg_{i+1}(f^n)} > 3N^{2N} 
\]
for all $0 \leq i \leq N-1$. We set
\[ g = f^n
\]
and claim that there exist a positive constant $C > 0$ and a non-empty Zariski open subset $U\subset X$ such that 
\[ h(g(R)) + h(g^{-1}(R)) \geq 3h(R) - C 
\] 
for all $R\in X(\qb)_{g} \cap X(\qb)_{g^{-1}} \cap U(\qb)$, where $h = h_H$. 
	
By Hironaka's theorem on resolution of indeterminacy, we may take $\pi: V \to X$ a composition of blow-ups (along smooth subvarieties of $X$), such that 
$\phi$ and $\psi$ are morphisms: 
\[ \xymatrix{ 
&  & V  \ar[d]_{\pi} \ar[dl]_{\psi} \ar[dr]^{\phi} &  \\
& X & X \ar@{-->}[l]^{g^{-1}} \ar@{-->}[r]_{g} & X. }
\] 
Consider the divisor  
\[ D = \phi^{\ast} H + \psi^{\ast} H - 3 \pi^{\ast} H.  
\] 
We claim that $D$ is big. By Siu's numerical criterion for bigness (\cite{Si93}, \cite[Theorem 2.2.15]{La04}), it is enough to show that 
\[ \left( \phi^{\ast} H + \psi^{\ast} H \right)^N > 3N \left( \phi^{\ast} H + \psi^{\ast} H \right)^{N-1} \cdot \pi^{\ast} H,  
\]
which is implied by
\[ \left( \phi^{\ast}H^p \cdot \psi^{\ast}H^{N-p} \right) > 3N \binom{N-1}{i}\left( \phi^{\ast}H^i \cdot \psi^{\ast}H^{N-1-i} \cdot \pi^{\ast} H \right)
\]
for each $0 \leq i \leq N-1$. Note that 
\[ \left( \phi^{\ast}H^p \cdot \psi^{\ast}H^{N-p} \right) = \deg_p(g^2) 
\]
and
\[ \binom{N-1}{i} \leq (N-1)^{N-1}. 
\]
Since $\phi^{\ast}H^i$ is a basepoint free class (see \cite[Definition 3.3.1, Theorem 3.3.3 (ii) and (iv)]{Da20}) and $\pi^{\ast}H$ is nef and big, by \cite[Corollary 3.4.5]{Da20}, we have 
\[ \phi^{\ast}H^i \leq (N-i+1)^i \left( \phi^{\ast}H^i \cdot \pi^{\ast}H^{N-i} \right) \pi^{\ast}H^i,  
\]
that is, 
\[ (N-i+1)^i \left( \phi^{\ast}H^i \cdot \pi^{\ast}H^{N-i} \right) \pi^{\ast}H^i - \phi^{\ast}H^i
\]
is a pseudo-effective class (see \cite[Definition 3.1.1]{Da20}), where $(N-i+1)^i \leq (N+1)^{N-1}$. Note that $\psi^{\ast}H^{N-1-i} \cdot \pi^{\ast} H$ is a nef class. Thus, 
\begin{equation*}
\begin{aligned}
&\left( \phi^{\ast}H^i \cdot \psi^{\ast}H^{N-1-i} \cdot \pi^{\ast} H \right) \\ 
\leq \ &(N+1)^{N-1} \left( \phi^{\ast}H^i \cdot \pi^{\ast}H^{N-i} \right) \left( \psi^{\ast}H^{N-1-i} \cdot  \pi^{\ast}H^{i+1} \right) \\ 
= \ &(N+1)^{N-1} \deg_i(g) \deg_{i+1}(g). 
\end{aligned} 
\end{equation*}
So it is enough to show that 
\[ \deg_p(g^2) > 3N (N^2-1)^{N-1} \deg_i(g) \deg_{i+1}(g) 
\]
for each $0 \leq i \leq N-1$. This is guaranteed by the choice of $g = f^n$.

Let $\B(D)$ denotes the stable base locus of $D$ (see \cite[Definition 2.13]{LS21}). Since $D$ is big, some positive multiple of $D$ is linearly equivalent to an effective divisor. In particular, $\B(D) \subsetneq V$ is a proper Zariski closed subset. We set 
\[ U = X \setminus \pi(\B(D)),  
\]
which is a non-empty Zariski open subset of $X$. 

For each $R\in X(\qb)_{g} \cap X(\qb)_{g^{-1}} \cap U(\qb)$, we take $\widetilde{R} \in V(\qb)$ such that $\pi(\widetilde{R}) = R$. From the functoriality of height functions, it follows that 
\[ h(g(R)) = h(\phi(\widetilde{R})) = h_{V, \phi^{\ast}H}(\widetilde{R}) + O(1), 
\]
\[ h(g^{-1}(R)) = h(\psi(\widetilde{R})) = h_{V, \psi^{\ast}H}(\widetilde{R}) + O(1), 
\]
and 
\[ h(R) = h(\pi(\widetilde{R})) = h_{V, \pi^{\ast}H}(\widetilde{R}) + O(1). 
\]
On the other hand, by the definition of $D$, we have 
\[ h_{V, \phi^{\ast}H}(\widetilde{R}) + h_{V, \psi^{\ast}H}(\widetilde{R}) = 3 h_{V, \pi^{\ast}H}(\widetilde{R}) + h_{V, D}(\widetilde{R}) + O(1). 
\]
Therefore, 
\[ h(g(R)) + h(g^{-1}(R)) = 3h(R) + h_{V, D}(\widetilde{R}) + O(1), 
\]
where these $O(1)$ are independent of the choice of $R$. We notice that $\widetilde{R} \notin \B(D)$ as $\pi(\widetilde{R}) = R \notin \pi(\B(D))$. Then by \cite[Lemma 2.26(1)]{LS21}, we have 
\[ h(g(R)) + h(g^{-1}(R)) \geq 3h(R) - C, 
\]
for all $R\in X(\qb)_{g} \cap X(\qb)_{g^{-1}} \cap U(\qb)$, where $C$ is independent of the choice of $R$. Clearly we may choose $C > 0$. This completes the proof. 
\end{proof}

\begin{proof}[Proof of Theorem \ref{bdd}]  By Lemmas \ref{aux} and \ref{tech}, there is a proper Zariski closed subset $Z \subsetneq X$, such that the set 
\[ \{ P \in X(\qb)_{f^n} \cap X(\qb)_{f^{-n}} : P \text{ is } f^n\text{-periodic and } \bo_{f^n}(P) \cap Z(\qb) = \varnothing \}
\] 
is a set of bounded height in $X(\qb)$ for some positive integer $n$. Clearly we have the following inclusion
\begin{equation*}
\begin{aligned}
&\{ P \in X(\qb)_{f} \cap X(\qb)_{f^{-1}} : P \text{ is } f\text{-periodic and } \bo_{f}(P) \cap Z(\qb) = \varnothing \} \\
\subset \ &\{ P \in X(\qb)_{f^n} \cap X(\qb)_{f^{-n}} : P \text{ is } f^n\text{-periodic and } \bo_{f^n}(P) \cap Z(\qb) = \varnothing \}. 
\end{aligned} 
\end{equation*}
Therefore, the assertion follows. 
\end{proof}

\begin{remark} Let $f^{\circ}: \bba^N \to \bba^N$ be an automorphism with the extension $f: \PP^N \dashrightarrow \PP^N$. Assume that $f$ is cohomologically hyperbolic. We specify $H = \PP^N \setminus \bba^N$ in the proof of Lemma \ref{aux}. It is unclear if $\pi(\B(D)) \subset \Supp H$ for the divisor $D$ constructed there. Thus, we do not know whether the set of all $f^{\circ}$-periodic points is a set of bounded height in $\bba^N(\qb)$. 
\end{remark}

Let $f: S \dashrightarrow S$ be a birational self-map of a smooth projective surface $S$. Then $f$ is cohomologically hyperbolic, if and only if $\delta_1(f) > 1$. Thus, Theorem \ref{bdd} can be applied to $f$ when $\delta_1(f) > 1$. We note that, when $f: S \to S$ is an automorphism with $\delta_1(f) > 1$, by \cite[Theorem 6.4]{Ka08}, the set of $f$-periodic points outside a \textit{$f$-periodic} proper Zariski closed subset of $S$ is a set of bounded height.

\begin{example} A smooth projective variety $X$ is called \textit{hyper-K\"{a}hler}, if its complex analytification is simply connected and $H^0(X, \Omega^2_X)$ is spanned by a symplectic form. By definition, a hyper-K\"{a}hler variety always has even dimension. Let $f: X \to X$ be an automorphism of a hyper-K\"{a}hler variety $X$. By \cite[Theorem 1.1]{Og09}, 
\[ \delta_{2N-k} (f) = \delta_k (f) = \delta_1(f)^k 
\]
for all $0 \leq k \leq N = \dim X /2$. Thus, $f$ is cohomologically hyperbolic, if and only if $\delta_1(f) > 1$. Analogously to \cite{Ka08}, the theory of canonical heights established in \cite[Theorem 2.27]{LS21} implies that the set of $f$-periodic points outside a $f$-periodic proper Zariski closed subset of $X$ is a set of bounded height. It is unknown whether a birational self-map of a hyper-K\"{a}hler variety is cohomologically hyperbolic provided the first dynamical degree greater than one . 
\end{example}

In general, it is difficult to compute higher dynamical degrees of a birational self-map $f$. But in dimension three, we have 
\[ \delta_2(f) = \delta_1(f^{-1}). 
\]
The following is an example of birational self-maps in dimension three that are cohomologically hyperbolic. Recall that for an automorphism $f^{\circ}: \bba^N \to \bba^N$ with the extension $f: \PP^N \dashrightarrow \PP^N$, we say $f^{\circ}$ (or $f$) is \textit{regular}, if $I(f) \cap I(f^{-1}) = \varnothing$, where $I(f)$ denotes the indeterminacy locus of $f$.

\begin{example}[{\cite[Example 2.5]{GS02}}] Consider the birational map $f: \PP^3 \dashrightarrow \PP^3$ induced by the polynomial automorphism 
\[ f^{\circ}: \bba^3 \dashrightarrow \bba^3, \ \ (x, y, z) \mapsto \left( yx^d + z, y^{d+1} + x, y \right), 
\]
where 
\[ (f^{\circ})^{-1}: \bba^3 \dashrightarrow \bba^3, \ \ (x, y, z) \mapsto \left(y - z^{d+1}, z, x - z(y - z^{d+1})^d\right). 
\]
Then 
\[ I(f) = \{ [ x : 0 : z : 0] \} \ \ \text{and} \ \ I(f^{-1}) = \{ [x : y : 0 : 0] \}.  
\]
Hence $f^{\circ}$ is not regular by definition as $I(f) \cap I(f^{-1}) \neq \varnothing$. One can see that $f^{-1}$ is weakly regular, so $f^{-1}$ is algebraically $1$-stable and 
\[ \delta_1(f^{-1}) = d^2 + d + 1. 
\]
(See \cite[Definition 2.1 and below]{GS02}.) On the other hand, $f$ is not weakly regular. But one can still see that $f$ is algebraically $1$-stable and 
\[ \delta_1(f) = d+1 
\]
as follows. Write 
\[ (f^{\circ})^n(x, y, z) = (x_n, y_n, z_n)
\]
for each $n \geq 1$. Then
\[ x_{n+1} = y_n x_n^d + z_n, \ \ y_{n+1} = y_n^{d+1} + x_n, \ \ \text{and} \ \ z_{n+1} = y_n. 
\]
We have 
\[ \deg x_1 = \deg y_1 = d + 1, \ \ \text{and} \ \ \deg z_1 = 1, 
\]
\[ \deg x_2 = \deg y_2 = (d + 1)^2, \ \ \text{and} \ \ \deg z_1 = d + 1. 
\]
By induction on $n$, it is not hard to see that 
\[ \deg x_n = \deg y_n = (d + 1)^n, \ \ \text{and} \ \ \deg z_n = (d + 1)^{n-1}. 
\]
for all $n \geq 1$. This shows the algebraic stability of $f$. Hence $f$ is cohomologically hyperbolic and one can apply Theorem \ref{bdd} to this $f$. 

We refer to \cite[Example 1.11]{GS02} and \cite[Theorem 4.1]{CG04} for more examples of automorphisms of $\bba^3$ that are cohomologically hyperbolic but not regular. 
\end{example}


\section{Proofs of remaining theorems}\label{thm1}

Now we start to prove Theorem \ref{dt}. Let $f: S \dashrightarrow S$ be a dominant rational self-map of a smooth projective surface $S$, and let $H$ be an ample divisor on $S$. Recall that the first dynamical degree $\delta_1(f)$ of $f$ is define as
\[ \delta_1(f) = \lim_{i \to +\infty}\left( \deg_{1, H} (f^i) \right)^{1/i},  
\]
where $\deg_{1, H} (f) = \left( f^{\ast}(H)\cdot H \right)$. The second dynamical degree $\delta_2(f)$ of $f$ is define as 
\[ \delta_2(f) = \lim_{i \to +\infty}\left( (f^i)^{\ast}(H^2) \right)^{1/i},  
\]
which is exactly the topological degree of $f$.

\begin{lemma}\label{hin} Let $f: S \dashrightarrow S$ be a dominant rational self-map of a smooth projective surface $S$. Fix an ample divisor $H$ on $S$. Then for every $\varepsilon \in (0, 1)$, there exists $n_{\varepsilon} \in \ZZ_{>0}$ such that for all $n \geq n_{\varepsilon}$, it holds 
	\[ \frac{h_H(f^{2n}(P))}{\delta_2(f^n)} + h_H(P) \geq \frac{\varepsilon^n \deg_{1, H}(f^n)}{\delta_2(f^n)} h_H(f^n(P)) - C_n
	\]
	for all $P \in S(\qb)_f \cap U_n(\qb)$, where $U_n \subset S$ is a non-empty open subset, and $C_n > 0$ is independent of the choice of $P$. 
\end{lemma}

\begin{proof} We write $h = h_H$ and $\deg_1 = \deg_{1, H}$ as $H$ is fixed. For an arbitrary $\varepsilon \in (0, 1)$, since 
	\[ \lim_{i\to +\infty} \left( \frac{\deg_1 (f^{2i})}{\varepsilon^i\deg_1 (f^{i})^2} \right)^{1/i} = \frac{\delta^2_1(f)}{\varepsilon\delta^2_1(f)} > 1, 
	\]
	there exists $n_{\varepsilon} \in \ZZ_{>0}$ such that for all $n \geq n_{\varepsilon}$, we have 
	\[ \frac{\deg_1 (f^{2n})}{\varepsilon^n\deg_1 (f^{n})^2} > 2.  
	\]
	Fix an arbitrary $n \geq n_{\varepsilon}$ and let
	\[ d_n = \deg_1 (f^n), \ \ \text{and} \ \ \delta_2 = \delta_2(f). 
	\]
	Then $\delta_2(f^n) = \delta_2^n$. 
	
	Take $\psi'_n: S' = S'_n \to S$ a composition of blow-ups, such that $\phi'_n$ is a morphism. Let $f'_n : S' \dashrightarrow S'$ be the induced dominant rational map, and take $\widetilde{\psi}_n: \widetilde{S} = \widetilde{S}_n \to S'$ a composition of blow-ups, such that $\widetilde{\phi}_n$ is a morphism. 
	\begin{equation}
		\xymatrix{ & & &\ts \ar[dl]_{\widetilde{\psi}_n} \ar[dr]^{\widetilde{\phi}_n} & & \\ 
			&  &S' \ar@{-->}[rr]^{f'_n} \ar[dl]_{\psi'_n} \ar[dr]^{\phi'_n} & &S' \ar[dl]_{\psi'_n} \ar[dr]^{\phi'_n} & \\
			& S \ar@{-->}[rr]_{f^{n}} & & S \ar@{-->}[rr]_{f^n} & & S. }
	\end{equation}
	We let 
	\[ \psi_n = \psi'_n \circ \widetilde{\psi}_n : \widetilde{S} \to S, \] 
	\[ \phi_n = \phi'_n \circ \widetilde{\phi}_n : \widetilde{S} \to S, \]  
	and 
	\[ \pi_n = \psi'_n \circ \widetilde{\phi}_n = \phi'_n \circ \widetilde{\psi}_n : \widetilde{S} \to S. 
	\]
	\begin{equation}
		\xymatrix{ 
			&  & \widetilde{S}  \ar[d]_{\pi_n} \ar[dl]_{\psi_n} \ar[dr]^{\phi_n} &  \\
			& S \ar@{-->}[r]_{f^{n}} & S \ar@{-->}[r]_{f^n} & S. }
	\end{equation}
	Consider the $\RR$-divisor 
	\[ D_n = \frac{\phi_n^{\ast} H }{\delta^n_2} + \psi_n^{\ast} H - \frac{\varepsilon^n d_n}{\delta^n_2} \pi_n^{\ast} H.  
	\] 
	We claim that $D_n$ is big. 
	
	Write $\pi = \pi_n$, $\phi = \phi_n$ and $\psi = \psi_n$ to ease the notation. By Siu's numerical criterion for bigness (\cite[Theorem 2.2.15]{La04}), it is enough to show that 
	\[ \left( \phi^{\ast} H + \delta^n_2 \psi^{\ast} H \right)^2 > 2 \varepsilon^n d_n \left( \phi^{\ast} H + \delta^n_2 \psi^{\ast} H \right) \cdot \pi^{\ast} H,  
	\] 
	which is implied by
	\[ \delta^n_2 \left( \phi^{\ast}H\cdot \psi^{\ast}H \right) > \varepsilon^n d_n \left( \phi^{\ast} H + \delta^n_2 \psi^{\ast} H \right) \cdot \pi^{\ast} H. 
	\]
	Notice that by the projection formula, we have 
	\[ \left( \phi^{\ast}H\cdot \psi^{\ast}H \right) = \left( \psi_{\ast}\phi^{\ast}H\cdot H \right) = \deg_1 (f^{2n}), 
	\] 
	and 
	\[ \left( \psi^{\ast} H \cdot \pi^{\ast} H \right) = \left( H \cdot \psi_{\ast} \pi^{\ast} H \right) = d_n,  
	\]
	because $\psi$ is birational. Let us compute 
	\[ \left( \phi^{\ast} H \cdot \pi^{\ast} H \right) = \left( \widetilde{\phi}_n^{\ast} \phi_n^{\prime\ast} H \cdot \widetilde{\phi}_n^{\ast} \psi_n^{\prime\ast}H \right) = \delta_2(f'_n) \left( \phi_n^{\prime\ast} H \cdot \psi_n^{\prime\ast}H \right) = \delta^n_2 d_n. 
	\]
	On the other hand, 
	\[ \frac{\deg_1 (f^{2n})}{d_n^2} = \frac{\deg_1 (f^{2n})}{\deg_1 (f^n)^2} > 2\varepsilon^n 
	\]
	as $n \geq n_{\varepsilon}$. This proves the bigness of $D_n$.

	We let 
	\[ U_n = S \setminus \psi(\B_{+}(D_n)),  
	\] 
	where $\B_{+}(D_n)$ denotes the augmented base locus of $D_n$ (\cite[Definition 2.13]{LS21}). Notice that $\B_{+}(D_n)$ is a proper Zariski closed subset of $\widetilde{S}$ because $D_n$ is big (see e.g., \cite[Lemma 2.14(2)]{LS21}). In particular, $U_n$ is a non-empty Zariski open subset of $S$.

	For every $P \in S(\qb)_f \cap U_n(\qb)$, we take $\widetilde{P} \in \widetilde{S}(\qb)$ such that $\psi(\widetilde{P}) = P$. Then $\pi(\widetilde{P}) = f^n(P)$ and $\phi(\widetilde{P}) = f^{2n}(P)$. From the functoriality of height functions, it follows that 
	\[ h(f^{2n}(P)) = h(\phi(\widetilde{P})) = h_{\widetilde{S}, \phi^{\ast}H}(\widetilde{P}) + O(1), 
	\]
	\[ h(P) = h(\psi(\widetilde{P})) = h_{\widetilde{S}, \psi^{\ast}H}(\widetilde{P}) + O(1), 
	\]
	and 
	\[ h(f^n(P)) = h(\pi(\widetilde{P})) = h_{\widetilde{S}, \pi^{\ast}H}(\widetilde{P}) + O(1). 
	\]
	On the other hand, by the definition of $D_n$, we have 
	\[ \frac{h_{\widetilde{S}, \phi^{\ast}H}(\widetilde{P})}{\delta^n_2} + h_{\widetilde{S}, \psi^{\ast}H}(\widetilde{P}) = \frac{\varepsilon^n d_n}{\delta^n_2} h_{\widetilde{S}, \pi^{\ast}H}(\widetilde{P}) + h_{\widetilde{S}, D_n}(\widetilde{P}) + O(1). 
	\]
	Therefore, 
	\[ \frac{h(f^{2n}(P))}{\delta^n_2} + h(P) = \frac{\varepsilon^n d_n}{\delta^n_2} h(f^{n}(P)) + h_{\widetilde{S}, D_n}(\widetilde{P}) + O(1), 
	\]
	where these $O(1)$ are independent of the choice of $P$. We notice that $\widetilde{P} \notin \B_{+}(D_n)$ as $\psi(\widetilde{P}) = P \notin \psi(\B_{+}(D_n))$. Then by \cite[Lemma 2.26(2)]{LS21}, we have 
	\[ \frac{h(f^{2n}(P))}{\delta^n_2} + h(P) \geq \frac{\varepsilon^n d_n}{\delta^n_2} h(f^{n}(P)) - C_n 
	\]
	for all $P \in S(\qb)_f \cap U_n(\qb)$, where $C_n$ is independent of the choice of $P$. Clearly we may choose $C_n > 0$. This completes the proof. 
\end{proof}

From Lemmas \ref{hin} and \ref{tec}, Theorem \ref{dt} follows easily.

\begin{proof}[Proof of Theorem \ref{dt}] Fixing an ample divisor $H$ on $S$, we write $h = h_H$ and $\deg_1 = \deg_{1, H}$. We let $\delta_1 = \delta_1(f)$ and $\delta_2 = \delta_2(f)$. Then $\delta_1 > \delta_2$ by assumption. 
	
	Fix $\varepsilon \in (\delta_2/\delta_1, 1)$. Since
	\[ \lim_{i\to +\infty} \left( \frac{\varepsilon^i \deg_1 (f^i)}{\delta^i_2} \right)^{1/i} = \frac{\varepsilon \delta_1}{\delta_2} > 1, 
	\]
	there exists $n'_{\varepsilon} \in \ZZ_{>0}$, such that 
	\begin{equation} 
		\varepsilon^n \deg_1 (f^n) > 2\delta^n_2 \geq \delta^n_2 + 1
	\end{equation}
	for all $n \geq n'_{\varepsilon}$. 
	 
	By Lemma \ref{hin}, there exists $n_{\varepsilon}\in \ZZ_{>0}$ such that for all $n \geq n_{\varepsilon}$, it holds that 
	\[ \frac{h(f^{2n}(R))}{\delta_2(f^n)} + h(R) \geq \frac{\varepsilon^n \deg_{1}(f^n)}{\delta_2(f^n)} h(f^n(R)) - C_n
	\]
	for all $R \in S(\qb)_f \cap U_n(\qb)$, where $U_n \subset S$ is a non-empty open subset, and $C_n > 0$ is independent of the choice of $R$.

	Consider a fixed $n \geq \max\{ n_{\varepsilon}, n'_{\varepsilon} \}$. Let $P \in S(\qb)_f$ with Zariski dense orbit $\OO_f(P)$ in $S$. Then $\OO_{f^n}(P)$ is Zariski dense in $S$ as well. By assumption and Lemma \ref{xi4.3}(1), Conjecture \ref{dml} holds for $(S, f^n)$. Then by Lemma \ref{obs}, there exists $k_0 \in \ZZ_{>0}$ such that $f^{nk}(P) \in U_n(\qb)$ for all $k \geq k_0$. Thus, $\OO_{f^n}(f^{nk_0}(P)) \subset U_n(\qb)$. From Lemma \ref{ks13} and Lemma \ref{tec}, it follows that  
	\[ \underline{\alpha}_{f^n}(P) = \underline{\alpha}_{f^n}(f^{nk_0}(P)) \geq \frac{\varepsilon^n}{2} \deg_1 (f^n). 
	\]
	Since $\underline{\alpha}_{f}(P)^n = \underline{\alpha}_{f^n}(P)$ by Lemma \ref{est}, we obtain that for all $n \geq \max\{ n_{\varepsilon}, n'_{\varepsilon} \}$, 
	\[ \underline{\alpha}_{f}(P) \geq \varepsilon \left( \frac{\deg_1 (f^n)}{2} \right)^{1/n}, 
	\]
	and hence 
	\[ \underline{\alpha}_{f}(P) \geq \varepsilon \lim_{n\to +\infty} \left( \frac{\deg_1 (f^n)}{2} \right)^{1/n} = \varepsilon \delta_1(f). 
	\]
	Since $\varepsilon \in (\delta_2/\delta_1, 1)$ is arbitrary, we conclude that 
	\begin{equation}
		\underline{\alpha}_{f}(P) \geq \delta_1(f). 
	\end{equation}
	On the other hand, $\overline{\alpha}_f(P) \leq \delta_1(f)$ holds in general by \cite[Theorem 1.4]{Ma20}. Therefore, $\underline{\alpha}_{f}(P) = \overline{\alpha}_{f}(P) = \delta_1(f)$, which means the existence of $\alpha_{f}(P)$ and the equality $\alpha_{f}(P) = \delta_1(f)$. 
\end{proof}

\begin{proof}[Proof of Theorem \ref{aff}] By Theorem \ref{dt}, we only need to show Conjecture \ref{dml} holds for $(S, f)$. For Case $(1)$ (resp. Case $(2)$), this follows from Lemma \ref{xi4.3}(2) and Theorem \ref{bgt1.3} (resp. Theorem \ref{xia}). 
\end{proof}

\begin{remark}\label{ex-jw} $(1)$ Let $f: \PP^2 \dashrightarrow \PP^2$ be a dominant rational self-map which is the extension of an affine plane endomorphism $\bba^2 \to \bba^2$, satisfying that $\delta_1(f) > \delta_2(f)$. Then the theory of canonical heights for $f$ was established in \cite[Main Theorem]{JW12}. More precisely, the limit 
\[ \widehat{h}(P) = \lim_{n\to +\infty} \frac{h(f^n(P))}{\delta_1(f)^n} 
\] 
exists, and is finite for each $P \in \bba^2(\qb)$. We have $\widehat{h} \not\equiv 0$ and $\widehat{h} \circ f = \delta_1(f)\widehat{h}$. If $\widehat{h}(P) = 0$, then $\overline{\alpha}_f(P) \leq \delta_2(f) < \delta_1(f)$. Moreover, if $f$ is an automorphism, then $\widehat{h}(P) = 0$ if and only if $P$ is $f$-periodic. Nevertheless, the last fact is unknown in general. Theorem \ref{aff} says that $\widehat{h}(P) > 0$ if $\OO_f(P)$ is Zariski dense in $\bba^2$. 

\medskip $(2)$ In \cite[Theorem D]{JR18}, the canonical height functions were constructed for certain birational self-maps of smooth projective surfaces. It is still unclear whether for a $\qb$-point with Zariski dense orbit, the value of the canonical height function is positive. We notice that this claim implies the arithmetic degree for this $\qb$-point equals the first dynamical degree of the birational self-map. 
\end{remark}

Now we show the existence of Zariski dense orbits. The following argument is inspired by \cite[Section 8]{JSXZ21}.

\begin{proof}[Proof of Theorem \ref{dt2}] $(1)$ Fixing an ample divisor $H$ on $S$, write $h = h_H$ and $\deg_1 = \deg_{1, H}$. Let $\delta_1 = \delta_1(f)$ and $\delta_2 = \delta_2(f)$. Then $\delta_1 > \delta_2$ by assumption. 

Take $\varepsilon_0 \in (\delta_2/\delta_1, 1)$. By the proof of Theorem \ref{dt}, for each $n \gg 0$, we have 
\begin{equation*} 
\varepsilon_0^n \deg_1 (f^n) > 2\delta^n_2 \geq \delta^n_2 + 1, 
\end{equation*}
and 
\[ \frac{h(f^{2n}(R))}{\delta_2(f^n)} + h(R) \geq \frac{\varepsilon_0^n \deg_{1}(f^n)}{\delta_2(f^n)} h(f^n(R)) - C_n
\]
for all $R \in S(\qb)_f \cap U_n(\qb)$, where $U_n \subset S$ is a non-empty open subset, and $C_n > 0$ is independent of the choice of $R$. 

Take $n_0 \gg 0$. Since $f$ has infinite order, by \cite[Theorem 1.7]{JX22} (see also \cite{Am11}), there exists a $\qb$-point $P \in S(\qb)_f$ such that $\OO_{f^{n_0}}(P)$ is infinite, and $\OO_{f^{n_0}}(P) \subset U_{n_0}(\qb)$. 
From Lemma \ref{est} and Lemma \ref{tec}, it follows that  
\[ \underline{\alpha}_{f}(P)^{n_0} = \underline{\alpha}_{f^{n_0}}(P) \geq \frac{\varepsilon_0^n \deg_1 (f^{n_0})}{2} > \delta^{n_0}_2. 
\]
Thus, $\underline{\alpha}_{f}(P) > \delta_2(f)$. 

Finally, we claim that $\OO_f(P)$ is Zariski dense in $S$. Let $\overline{\OO_f(P)}$ be the Zariski closure of $\OO_f(P)$ in $S$. By \cite[Lemma 2.7]{MMSZ20}, $f^t(P)$ is contained in an $f^s$-invariant irreducible component $V$ of $\overline{\OO_f(P)}$ for some $s, t \geq 1$. Suppose that $V \subsetneq S$. Then we have $\dim V = 1$. Notice that 
\[ \underline{\alpha}_{f^s}(f^t(P)) = \underline{\alpha}_{f^s|_V}(f^t(P)) \leq \delta_1(f^s|_V) \leq \delta_2(f^s), 
\]
where the equality follows from \cite[Lemma 2.5]{MMSZ20}, the first inequality follows from \cite[Proposition 3.11]{JSXZ21}, and the second inequality follows from the fact that $\delta_1(f^s|_V)$ (resp. $\delta_2(f^s)$) is the topological degree of the map $f^s|_V$ (resp. $f^s$). By $\underline{\alpha}_{f^s}(f^t(P)) = \underline{\alpha}_{f^s}(P) = \underline{\alpha}_{f}(P)^s$ and $\delta_2(f^s) = \delta_2(f)^s$, we conclude that $\underline{\alpha}_{f}(P) \leq \delta_2(f)$. This is a contradiction. Therefore, $\OO_f(P)$ is Zariski dense in $S$. 

\medskip $(2)$ By Claim (1), there exists a $\qb$-point $Q \in S(\qb)_f$ with Zariski dense forward $f$-orbit $\OO_f(Q)$ in $S$. Since $W \subset S$ is a non-empty Zariski open subset, there is a positive integer $k$ such that $f^k(Q) \in W(\qb)$. Take $P = f^k(Q) \in W(\qb)$ and then $\OO_f(P)$ is Zariski dense in $W$. 
\end{proof}

\end{document}